\documentclass[12pt,a4paper,twoside, english]{amsart}
\usepackage[a4paper, left=32mm, right=32mm, top=36mm, bottom=36mm]{geometry}
\usepackage{amsmath,amsfonts,amssymb}
\usepackage{float}
\usepackage{babel}
\usepackage{hyperref}

\usepackage{wrapfig}
\usepackage{booktabs}

\usepackage{stmaryrd}
\usepackage{comment}
\usepackage{caption}
\usepackage{sidecap}
\usepackage{xcolor}

\hypersetup {colorlinks=true, linkcolor=blue, urlcolor=blue, citecolor=blue}

\newcommand\R{\mathbb{R}}
\newcommand\C{\mathbb{C}}

\newcommand\eps{\varepsilon}
\DeclareMathOperator{\curl}{\mathrm{curl}}

\newcommand{\zvec}[2]{\left(\begin{matrix}#1\cr #2\end{matrix}\right)}

\newcommand\cT{\mathcal{T}}

\renewcommand{\Re}{\operatorname{Re}}
\renewcommand{\Im}{\operatorname{Im}}

\newcommand{\del}{\partial}
\newcommand{\ra}{\rangle}
\newcommand{\la}{\langle}

\newcommand{\id}{\mathrm{id}}

\newcommand{\loc}{\mathrm{loc}}

\newcommand\fhi{\varphi}
\newcommand{\weakto}{\rightharpoonup}

\newcommand{\nn}[1]{\nu \times #1|_{\Gamma}}
\newcommand{\nnn}[1]{((#1 \times \nu)|_{\Gamma} \times \nu)}
\newcommand{\cptemb}{\xhookrightarrow{\text{cpt. }}}

\usepackage{mathtools}
\usepackage{extarrows}
\renewcommand{\div}{\operatorname{div}}

\newtheorem{theorem}{Theorem}[section]
 
\newtheorem{lemma}[theorem]{Lemma}
\newtheorem{corollary}[theorem]{Corollar}

\newtheorem{remark}[theorem]{Remark}

\newtheorem{assumption}[theorem]{Assumption}

\def\XXint#1#2#3{{\setbox0=\hbox{$#1{#2#3}{\int}$ }
		\vcenter{\hbox{$#2#3$ }}\kern-.6\wd0}}

\numberwithin{equation}{section}

\allowdisplaybreaks

\begin{document}
\bibliographystyle{abbrv}
	
\pagestyle{myheadings} \markboth{Maxwell's equations with mixed
  impedance boundary conditions}{B.~Schweizer and D.~Wiedemann}
	
\thispagestyle{empty}
\begin{center}
  ~\vskip3mm {\Large\bf Maxwell's equations with mixed\\[2mm]
    impedance
    boundary  conditions}\\[7mm]
  {\large B.~Schweizer\footnotemark[1] and
    D.~Wiedemann\footnotemark[1]}\\[4mm]
  
  October 16, 2025 \\[2mm]
  
\end{center}

\footnotetext[1]{Technische Universität Dortmund, Fakult\"at f\"ur
  Mathematik, Vogelspothsweg 87, D-44227 Dortmund,
  ben.schweizer$@$tu-dortmund.de, david.wiedemann$@$tu-dortmund.de}
	
\begin{center}
  \vskip4mm
  \begin{minipage}[c]{0.87\textwidth}
    {\bf Abstract:} We study the time-harmonic Maxwell equations on
    bounded Lipschitz domains with an impedance boundary
    condition. The impedance coefficient can be matrix valued such
    that, in particular, a polarization dependent impedance is
    modeled. We derive a Fredholm alternative for this system. As a
    consequence, we obtain the existence of weak solutions for
    arbitrary sources when the frequency is not a resonance
    frequency. Our analysis covers the case of singular impedance
    coefficients.


    

  \end{minipage}\\[6mm]
\end{center}

\section{Introduction}

We study the time-harmonic Maxwell equations in a bounded Lipschitz
domain. Our interest is to investigate an impedance boundary
condition, a condition that can be compared with a Robin boundary
condition in a scalar problem. The impedance coefficient $\Lambda$ can
be matrix valued and can, therefore, model a polarization dependent
impedance. Furthermore, the matrix coefficient may be singular in the
sense that it is non-trivial, but it vanishes on a non-trivial
subspace of the tangent space. Our result is a Fredholm alternative
for this Maxwell system.

Let us describe the system in mathematical terms. Given is a bounded
Lipschitz domain $\Omega\subset \R^3$, two coefficient functions
$\eps \in L^\infty(\Omega, \C^{3 \times 3})$ and
$\mu \in L^\infty(\Omega, \C^{3 \times 3})$, a frequency $\omega>0$
and right-hand sides $f_h, f_e \colon \Omega\to \C^3$ of class
$L^2$. We seek for functions $E, H \colon \Omega\to \C^3$ that
satisfy, in $\Omega$,
\begin{subequations}\label{eq:Maxwell-strong-E-H}
  \begin{eqnarray} \label{eq:Maxwell-strong-E-H:1}
    \curl E &= i \omega \mu H + f_h \,,\\
    \label{eq:Maxwell-strong-E-H:2}
    \curl H &= - i \omega \varepsilon E +f_e  \,.
  \end{eqnarray}
  The system is complemented with the tangential boundary condition
  \begin{equation}
    \label{eq:Maxwell-strong-E-H:3}
    E \times \nu = \Lambda ((H \times \nu) \times \nu)
    \qquad\textrm{ on } \Gamma \coloneqq \partial\Omega \,,
  \end{equation}
\end{subequations}
where $\nu$ is the exterior normal vector on $\del\Omega$ and
$\Lambda$ is a matrix valued impedance coefficient. The normal vector
is a map $\nu\colon \del\Omega = \Gamma\ni x\mapsto \nu(x) \in \R^3$,
defined for almost every $x$ (almost every $x$ in the sense of the
two-dimensional measure on $\Gamma$). Similarly, $\Lambda$ is a map
$\Lambda\colon \Gamma\ni x\mapsto \Lambda(x) \in \C^{3\times 3}$.  The
map $H(x) \mapsto -(H(x) \times \nu(x)) \times \nu(x)$ is the
projection onto the tangential space $T_x\Gamma$.

We note that the above setting covers perfect conductor boundary
conditions on some part of the boundary (setting $\Lambda = 0$ on this
part) and impedance boundary conditions in the remaining part of the
boundary. In view of such applications, it is important that we do not
impose a continuity property on the coefficient $\Lambda$.  The
following situation is also covered: Along the boundary (or a certain
part of the boundary), there is a perfect reflection condition for
some polarization direction and an impedance condition for the
orthogonal polarization direction. This is modeled with a singular map
$\Lambda\neq 0$.

The solution space and the weak solution concept are defined below in
\eqref{eq:HGL2-curl-M} and \eqref{eq:WeakForm:Maxwell-in-E}. The weak
form encodes the boundary condition with matrix valued functions $\Sigma$ and
$\Theta$ instead of $\Lambda$. We will discuss that the two
formulations are equivalent, see Section \ref{sec:BoundaryCond} and
\eqref{eq:BC-in-E}.  Our main theorem is formulated with assumptions
on $\Sigma$ and $\Theta$, but we provide also a formulation of the
assumptions in terms of $\Lambda$, see Lemma
\ref{lem:AssumptionsLambda}.

\begin{assumption}[Assumptions on the
  coefficients]\label{ass:Coefficents}
  The bulk coefficients are maps
  $\eps, \mu \in L^\infty(\Omega, \C^{3 \times 3})$. They are coercive
  in the sense that, for some constant $c_0 > 0$, for almost every
  $x\in \Omega$, there holds
  \begin{equation}\label{eq:Coercivity:eps-mu}
    \bar\zeta \cdot \eps(x)\zeta \geq c_0 \|\zeta\|^2\quad
    \text{ and }\quad
    \bar\zeta \cdot \mu(x)\zeta \geq c_0 \|\zeta\|^2
    \qquad \text{for all } \zeta \in \C^3\,.
  \end{equation}
  The boundary coefficients are given by maps
  $\Theta, \Sigma \in L^\infty(\Gamma, \C^{3 \times 3})$. Their sum is
  coercive: There exists a constant $c_0 > 0$ such that, for almost
  every $x\in \Gamma$,
  \begin{eqnarray}\label{eq:Coerc:Sigma+Theta}
    \bar\zeta \cdot (\Sigma(x) + \Theta(x)) \zeta \geq c_0 \|\zeta\|^2
    \qquad \text{ for all } \zeta \in \C^3\,.
  \end{eqnarray}
\end{assumption}

With the coercivity requirement in \eqref{eq:Coercivity:eps-mu} and
\eqref{eq:Coerc:Sigma+Theta} we demand, in particular, that the
left-hand side is real for all arguments $x$ and $\zeta$. Every real
valued, symmetric and coercive matrix is also coercive in the above
sense.  With $\mu\in L^\infty$ coercive in the above sense, also the
inverse matrix $\mu^{-1}$ is coercive in the above sense. We choose
$c_0>0$ such that $\eps$, $\eps^{-1}$, $\mu$, $\mu^{-1}$ are coercive
with this constant.  Regarding \eqref{eq:Coerc:Sigma+Theta}, we
mention that it would be sufficient to consider only
$\zeta \in T_x\Gamma$, where we understand $T_x\Gamma$ as complex
vector space.

We note that not only real matrices are coercive in the above sense. A
two-dimensional example is given by the matrix {\small
  $\zvec{2 &-i}{i &2} \in \C^{2\times 2}$}.

\smallskip
We should clarify how $\Sigma$ and $\Theta$ can be chosen for a given
map $\Lambda$. This helps to derive conditions on $\Lambda$ such that
$\Sigma$ and $\Theta$ satisfy the above conditions.

\begin{lemma}[Assumptions in terms of
  $\Lambda$]\label{lem:AssumptionsLambda}
  Let a boundary condition be given by a map
  $\Lambda \in L^\infty(\Gamma, \C^{3 \times 3})$. For almost every
  $x\in \Gamma$, with the kernel $Z = \ker(\Lambda(x))\subset \C^3$
  and its orthogonal complement $Z^\perp\subset \C^3$, we assume: (i)
  $\nu(x) \in Z$ such that $Z^\perp\subset T_x\Gamma$, (ii) $\Lambda$
  maps into the orthogonal complement of its kernel,
  $\operatorname{R}(\Lambda(x)) \subset Z^\perp$, (iii) $\Lambda(x)$
  is coercive on $Z^\perp$: For a constant $c_0 >0$ holds, for almost
  every $x\in \Gamma$ and for the space $Z$ corresponding to $x$:
  \begin{eqnarray}\label{eq:Coerc:Sigma+M}
    \bar\zeta \cdot \Lambda(x) \zeta
    \geq c_0 \|\zeta\|^2 \qquad
    \text{for all } \zeta \in Z^\perp\,.
  \end{eqnarray}

  In this situation, we set $\Theta(x) \coloneqq \Pi_Z$, the
  orthogonal projection onto $Z$. We define
  $\Sigma(x) \in L^\infty(\Omega, \C^{3 \times 3})$ by demanding that
  it is the inverse of $\Lambda(x)|_{Z^\perp} \colon Z^\perp\to Z^\perp$ on
  $Z^\perp$ and that $\Sigma(x)|_Z = 0$.  Then, $\Sigma$ and $\Theta$
  have the properties of Assumption \ref{ass:Coefficents}. The weak
  solution concept with $\Sigma$ and $\Theta$ encodes the strong
  formulation of \eqref{eq:Maxwell-strong-E-H:3}.
\end{lemma}

The proof of Lemma \ref{lem:AssumptionsLambda} is given in Section
\ref{sec:BoundaryCond}.

\begin{remark}[Special case $\Lambda=0$]
  We emphasize that our setting allows to choose $\Lambda \equiv
  0$. This choice models a perfectly conducting boundary. For
  $\Lambda \equiv 0$, we can set $\Theta = \id$, the space
  $H_\Theta(\curl, \Omega, \Gamma)$ coincides with the classical space
  $H_0(\curl, \Omega)$ and all boundary integrals in the proofs are
  vanishing. Conceptually, our approach for general $\Lambda$ is not
  more involved than classical existence proofs for
  $\Lambda \equiv 0$.
\end{remark}

\subsection{Function spaces and weak formulation}

A fundamental function space in the analysis of Maxwell's equations is
\begin{eqnarray}\nonumber
  H(\curl,\Omega) \coloneqq \left\{u \in L^2(\Omega,\C^3) \ \middle|\ 
  \exists f \in L^2(\Omega, \C^3)\colon \phantom{\int}\right.
  \qquad\qquad
  \\
    \left. \int_\Omega f \cdot \phi = \int_\Omega u \cdot \curl \phi \
  \forall \phi \in C^\infty_c(\Omega, \C^3) \right\}\,.
  \label{eq:H-curl}
\end{eqnarray}
The function $f$ is the distributional curl of $u$ and we therefore
write $\curl u = f$.  The space $H(\curl, \Omega)$ is a Hilbert space
with
$\|u\|_{H(\curl, \Omega)}^2 \coloneqq \int_\Omega \left\{|u|^2 +
  |\curl u|^2\right\}$ and the scalar product
$\la u, \fhi\ra \coloneqq \la u, \fhi\ra_{L^2(\Omega)} + \la \curl u,
\curl \fhi\ra_{L^2(\Omega)}$.

Functions in $H(\curl, \Omega)$ have a tangential trace in the
distributional sense, but we do not need the theory of tangential
traces here. We construct a space of functions with a tangential trace
in $L^2(\Gamma)$ as follows:
\begin{eqnarray}
  \nonumber
  H(\curl, \Omega, \Gamma) \coloneqq
  \left\{u \in H(\curl,\Omega)\ \middle|\
  \exists g \in L^2(\Gamma, \C^3) \colon \phantom{\int}
  \right.
  \qquad\qquad\qquad\qquad
  \\
  \left. \int_\Omega \left\{\curl u \cdot \phi - u \cdot \curl \phi \right\}
  = \int_{\Gamma} g \cdot \phi \quad \forall \phi \in H^1(\Omega)\right\}\,.
  \label{eq:HGL2-curl}
\end{eqnarray}
The function $g$ in \eqref{eq:HGL2-curl} is the tangential trace, we
write $\nn{u} \coloneqq g$. We remark that a tangential trace function
$g$ satisfies always $g(x)\cdot \nu(x) = 0$ for a.e.~$x\in \Gamma$; this
can be seen by inserting test-functions $\phi$ that point in normal
direction along $\Gamma$.  The space is a Hilbert space with the norm
$\|u\|_{H(\curl, \Omega, \Gamma)}^2 \coloneqq \int_\Omega \left\{|u|^2
  + |\curl u|^2\right\}+ \int_{\Gamma} |\nn{u}|^2$.

For our application, we define the subspace of functions $u$ with
$\Theta (\nn{u}) =0$:
\begin{eqnarray}
  H_\Theta(\curl, \Omega, \Gamma) \coloneqq
  \left\{u \in H(\curl,\Omega, \Gamma) \mid \Theta(\nn{u})=0 \right\}\,.
  \label{eq:HGL2-curl-M}
\end{eqnarray}

\medskip {\bf Weak form of the Maxwell system.} On $\Omega$ with
boundary $\Gamma = \del\Omega$, we study the following problem: Find
$E \in H_\Theta(\curl, \Omega, \Gamma)$ such that
\begin{align}
  \begin{aligned}\label{eq:WeakForm:Maxwell-in-E}
    &\int\limits_\Omega \left\{\mu^{-1} \curl E \cdot \curl \phi
      -\omega^2 \eps E \cdot \phi\right\} - i \omega
    \int\limits_{\Gamma} \Sigma (\nn{E}) \cdot \nn{\phi}
    \\
    &\qquad = \int\limits_\Omega \left\{i \omega f_e \cdot \phi
      + \mu^{-1} f_h \cdot \curl \phi \right\} \quad \forall\ \phi \in
    H_\Theta(\curl, \Omega, \Gamma)\,,
  \end{aligned}
\end{align} 
For a solution $E\in H_\Theta(\curl, \Omega, \Gamma)$ of \eqref{eq:WeakForm:Maxwell-in-E} we use
$H\coloneqq (i \omega \mu)^{-1} \left(\curl E - f_h\right)$ as the
corresponding magnetic field.

Our main result is the following:

\begin{theorem}[Fredholm alternative]\label{thm:FredholmAlternative}
  Let $\Omega\subset\R^3$ be a bounded Lipschitz domain, let
  $\omega>0$ be a frequency and let $\eps, \mu$ and $\Sigma, \Theta$
  be coefficients that satisfy Assumption \ref{ass:Coefficents}.
  Then, the problem \eqref{eq:WeakForm:Maxwell-in-E} satisfies a
  Fredholm alternative: Either (i) for $f_h, f_e =0$ system
  \eqref{eq:WeakForm:Maxwell-in-E} has a non-trivial solution, or (ii)
  for every $f_h, f_e \in L^2(\Omega, \C^3)$, system
  \eqref{eq:WeakForm:Maxwell-in-E} has a weak solution
  $E \in H_{\Theta}(\curl,\Omega, \Gamma)$.
\end{theorem}

The reader might prefer the following formulation, which is an
immediate consequence of Theorem \ref{thm:FredholmAlternative}.

\begin{corollary}[Existence and uniqueness]\label{cor:existence}
  Let $\Omega\subset\R^3$ be a bounded Lipschitz domain and let
  $\eps, \mu$ and $\Sigma, \Theta$ satisfy Assumption
  \ref{ass:Coefficents}.  Let $\omega>0$ be a frequency such that
  system \eqref{eq:WeakForm:Maxwell-in-E} with $f_h = 0$ and $f_e = 0$
  has only the trivial solution.  Then, for every
  $f_h, f_e \in L^2(\Omega, \C^3)$, system
  \eqref{eq:WeakForm:Maxwell-in-E} has a unique weak solution
  $E \in H_{\Theta}(\curl,\Omega, \Gamma)$.
\end{corollary}

The corollary is a consequence of Theorem \ref{thm:FredholmAlternative}, but it can also be obtained with a limiting
absorption principle. We provide such a proof in Section
\ref{sec:LimitingAbs}.

\subsection{Overview of the available literature}

We consider \eqref{eq:Maxwell-strong-E-H} on a bounded domain, this
setting is often denoted as the cavity problem. When the domain of
interest is unbounded, one has to impose radiation conditions at
infinity to have a well-posed problem. For recent well-posedness
result in the unbounded domain of a waveguide, we refer to \cite{Kirsch-Schweizer-ARMA-2025, HalfWG-Maxwell-LSS-2025} and references
therein. Radiation conditions for the Helmholtz equation in exterior
domains have a long history, variable coefficients have been treated
in \cite{Jaeger-1967}.

For $\Lambda=0$ the boundary condition \eqref{eq:Maxwell-strong-E-H:3}
simplifies to $\nn{E}=0$ and models a perfectly conducting material in
some exterior medium $\Omega'$ with
$\partial \Omega \subset \partial \Omega'$. If the exterior medium
$\Omega'$ is dissipative, as for instance if it is a good but not
perfect conductor, one often uses an impedance boundary condition with
$\Lambda \neq 0$. A possibility is to define $\Lambda$ as the
multiplication with a positive number, this is the usual choice for
time-harmonic fields. The impedance boundary condition can also be
used to approximate the Silver--Müller radiation condition. We refer
to \cite[Chapter 1.6.1]{ACL18} for a more detailed discussion on the
application of the boundary conditions.

In the case $\Lambda = 0$, the variational approach to the cavity
problem leads to a non-coercive sesquilinear form and an existence
result can only be formulated as a Fredholm alternative.  The Fredholm
alternative can be derived by a well-established approach (see for
instance \cite{ACL18, KH15, Mon03}), we also follow this approach in
Section \ref{sec:Fredholm}. It consists in the following three steps:
(i) A Helmholtz decomposition of the data yields a control of the
divergence of the unknown (ii) derivation of a compact embedding of
the solution space into $L^2(\Omega, \C^3)$ (iii) reformulating the
weak form of the problem in terms of operators, the compact embedding
provides the Fredholm property for one of the operators. The
derivation of the Fredholm alternative for $\Lambda \neq 0$ requires
only an appropriate formulation, an adjustment of the function spaces
and the sesquilinear form and a refinement of the compactness result
of $(ii)$. For $\Lambda$ equal to a positive constant, proofs can be
found in \cite{Mon03} and \cite{ACL18}.

An alternative approach for the derivation of a weaker existence
result in the sense of Corollary \ref{cor:existence} is the Eidus
principle of limiting absorption \cite{Eui62} (cf.~\cite[Chapter 8]{Urb20}), which we also employ in Section \ref{sec:LimitingAbs}. The
limiting absorption principle replaces Step (iii) of the first
approach but does still rely on the Helmholtz decomposition and the
compactness result (Steps (i) and (ii)). We emphasize that the
limiting absorption principle is more than just a method of proof: It
provides an additional information, namely the convergence of
solutions for a vanishing damping parameter.

The Helmholtz decomposition can be derived by elementary methods, we
present this in Section \ref{sec:Helmholtz}. The compact embedding of
the solution space is more technical and has a longer history:
Elementary calculations provide
$H(\curl,\Omega) \cap H(\div, \Omega)\subset H^1_\loc(\Omega, \C^3)$
and, thus, by Rellich's compact embedding
$H(\curl,\Omega) \cap H(\div, \Omega) \cptemb L^2_\loc(\Omega,
\C^3)$. However, the compact embedding holds only locally and
$H(\curl,\Omega) \cap H(\div, \Omega)$ cannot be compactly embedded in
$L^2(\Omega, \C^3)$, even for smooth domains \cite{ABDG98, Mur79}.
For $C^{1,1}$ domains, the Gaffney--Friedrichs inequality gives a
continuous embedding of the subspace of functions with vanishing
tangential boundary values, $H_0(\curl, \Omega) \cap H(\div, \Omega)$
or vanishing normal boundary values,
$H(\curl, \Omega) \cap H_0(\div, \Omega)$, in $H^1(\Omega,
\C^3)$. Then, Rellich's compactness theorem provides the desired
compactness of these spaces in $L^2(\Omega, \C^3)$. The
Gaffney--Friedrichs inequality was shown for smooth domains in
\cite{Fri55, Gaf51, Gaf55, Mor66}. For convex domains,
Gaffney--Friedrichs inequality was shown for vanishing tangential
boundary values in \cite{Ned82} and for vanishing normal boundary
values in \cite{Sar82}. 
For $C^{1,1}$-domains, the Gaffney--Friedrichs inequality was derived
in \cite{DL72, FT78, GR86} for vanishing tangential components and in
\cite{ABDG98, Cos91} for vanishing normal components.  For
Gaffney--Friedrichs inequality with more general boundary conditions
see \cite{Csa12}.  However, this inequality does not hold for
arbitrary Lipschitz domains (see \cite{ABDG98}) and a different
approach is required for the derivation of Maxwell's compactness
theorem.

For a class of piecewise smooth domains the compact embedding of
$H_0(\curl, \Omega) \cap H(\div, \Omega)$ and
$H(\curl, \Omega) \cap H_0(\div, \Omega)$ in $L^2(\Omega)$ was shown
in \cite{Wec74} and for general Lipschitz domains in \cite{Pic84, Web80}. This approach relies essentially on the construction of
suitable scalar and vector potentials.  The compactness result was
improved to a more quantitative estimate, namely the continuous
embedding in $H^{1/2}(\Omega, \C^3)$ using additionally regularity
results for the Dirichlet- and Neumann-problem, see also
\cite{Mon03}. The regularity results are based on the non-tangential
maximal functions \cite{JK81, JK82} and enable also inhomogeneous
$L^2(\Gamma)$-regular tangential or normal boundary values. The case
of inhomogeneous boundary value becomes highly relevant in the
analysis for Maxwell's equations with an impedance boundary condition.
More elementary regularity arguments provide an embedding in
$H^{1/2-\delta}(\Omega, \C^3)$ for $\delta >0$, see \cite{Cos88}; this
is sufficient for the desired compactness.

The compactness result is extended to mixed boundary values, where at
some part of the boundary the tangential trace vanishes while on the
remaining part the normal trace vanishes \cite{BPS16, Joc97}; an
extension to inhomogeneous $L^2(\Gamma)$-regular mixed boundary data
is presented in \cite{PS23}.  Related discussions on vector potentials
with mixed boundary conditions are presented in \cite{AB21}

\smallskip On polarization dependent boundary conditions: Variational
formulations of the time-harmonic Maxwell equations have mainly
addressed polarization independent boundary conditions. In
\cite{SIMA25} the homogenization of a thin layer of perfect conductors
is considered, it can lead to a polarization dependent interface
condition, which is strongly related to a polarization dependent
boundary conditions. An existence result for this kind of interface
condition is presented in \cite{BSW25}.  The novelty of the present
work is that it combines two qualitatively different boundary
conditions on the same part of the boundary, namely the reflection
boundary condition is some direction and an impedance boundary
condition in another direction.

\subsection{Organization of this text}
In Section \ref{sec:BoundaryCond}, we discuss the weak solution
concept for \eqref{eq:Maxwell-strong-E-H}. In particular, we discuss
the different formulations of the boundary condition and the
equivalence of these formulations under reasonable assumptions.
Section \ref{sec:Helmholtz} is devoted to Helmholtz decompositions
which allows us to simplify the problem: It is sufficient to consider
solutions and test-functions in a space of divergence-free
functions. The proof of Theorem \ref{thm:FredholmAlternative} is given
in Section \ref{sec:Fredholm}.  We make use of a well-known
compactness result for functions with bounded divergence and curl; in
order to have this exposition self-contained, we include the proof of
the compactness statement in Section \ref{sec.compactness}.  The same
compactness statement is also used in the limiting absorption
principle that is presented in Section \ref{sec:LimitingAbs}.  It
provides another proof of Corollary \ref{cor:existence}.

\section{Discussion of weak formulation and boundary
  conditions}\label{sec:BoundaryCond}

In order to motivate the weak solution concept of this article, let us
consider a weak solution $(E,H)$. A weak solution is given by
$E \in H_\Theta(\curl, \Omega, \Gamma)$ that satisfies \eqref{eq:WeakForm:Maxwell-in-E}, the magnetic field is set to
$H = (i \omega \mu)^{-1} \left(\curl E - f_h\right)$.

The strong equation \eqref{eq:Maxwell-strong-E-H:1} is satisfied by
the definition of $H$.  We use this relation for $H$ to substitute
$\curl E$ in \eqref{eq:WeakForm:Maxwell-in-E} and find
\begin{align}\label{eq:WeakFrom->MixedForm:1}
  &\int\limits_\Omega \left\{ i \omega H \cdot \curl \phi
    -\omega^2 \eps E \cdot \phi\right\}
    - i \omega \int\limits_{\Gamma} \Sigma(\nn{E}) \cdot \nn{\phi}
    =
    \int\limits_\Omega i \omega f_e \cdot \phi
\end{align}
for all $\phi \in H_\Theta(\curl, \Omega, \Gamma)$.  For
test-functions $\phi \in C^\infty_c(\Omega, \C^3)$ in
\eqref{eq:WeakFrom->MixedForm:1}, the boundary integral vanishes and
we obtain \eqref{eq:Maxwell-strong-E-H:2}.

It remains to check the boundary conditions.  For this step, we assume
that the solution has additional regularity such that boundary traces
are well defined. We consider an arbitrary
$\phi \in H_\Theta(\curl, \Omega, \Gamma)$ in
\eqref{eq:WeakFrom->MixedForm:1}, integrate the first term by parts
and insert \eqref{eq:Maxwell-strong-E-H:2}. This provides
\begin{align}\label{eq:WeakImpedanceRegular}
  \int\limits_{\Gamma} i \omega H \cdot \nn{\phi}
  - i \omega \int\limits_{\Gamma} \Sigma(\nn{E}) \cdot \nn{\phi}
  =
  0\,.
\end{align}
Therefore, the weak solution satisfies pointwise the boundary
conditions
\begin{subequations}
  \label{eq:BC-in-E}
  \begin{align}
    \label{eq:BC-in-E-1}
    &\Theta(\nu \times E)= 0 \,,\\
    \label{eq:BC-in-E-2}
    &[ H - \Sigma (\nu \times E)] \cdot (\nu \times \phi) = 0
      \quad \forall \phi  \text{ with } \Theta(\nu \times \phi) =0\,,
  \end{align}
\end{subequations}
where the first condition follows from the fact that we seek for a
solution $E \in H_\Theta(\curl, \Omega, \Gamma)$.  The second
condition, interpreted pointwise, implies that
$H - \Sigma (\nu \times E)$ is orthogonal to the kernel of $\Theta$ in
the tangential space.

Let us consider the situation of Lemma \ref{lem:AssumptionsLambda}
and a point $x\in \Gamma$. We claim that $(E,H)$ satisfies the
boundary condition \eqref{eq:Maxwell-strong-E-H:3} if and only if it
satisfies \eqref{eq:BC-in-E}. To show one implication, let $(E,H)$
satisfy \eqref{eq:BC-in-E}. We use the orthogonal projection $\Pi_x$
to the tangent space $T_x\Gamma$. By \eqref{eq:BC-in-E-2}, the
expression $\Pi_x [H - \Sigma (\nu \times E)](x)$ is orthogonal to the
kernel of $\Theta(x)$. This means that it is an element of $Z$, the
kernel of $\Lambda(x)$. Supressing the point $x$, we obtain
$\Lambda \Pi H = \Lambda\Pi \Sigma (\nu \times E)$.  The left-hand
side is identical to $-\Lambda((H\times\nu)\times\nu)$. On the
right-hand side, $\Pi$ acts trivially on $\Sigma (\nu \times E)$,
since the latter is a tangential vector. Since $\Lambda$ is the
inverse of $\Sigma$ on the kernel of $\Theta$ and $\nu \times E$ is in
this kernel, the right-hand side is $\nu \times E$.  We have therefore
concluded \eqref{eq:Maxwell-strong-E-H:3}; we note that we use in
\eqref{eq:Maxwell-strong-E-H:3} the more standard notation where the
normal vector is always written behind the fields.

Vice versa, let $(E,H)$ satisfy \eqref{eq:Maxwell-strong-E-H:3}. We
apply $\Theta$ on \eqref{eq:Maxwell-strong-E-H:3} and obtain from
$\Theta \circ \Lambda = 0$ that $\Theta(E \times \nu) = 0$, which
shows \eqref{eq:BC-in-E-1}. To deduce \eqref{eq:BC-in-E-2}, we apply
$\Sigma$ on \eqref{eq:Maxwell-strong-E-H:3} and multiply the equation
by $\phi \times \nu$ for $\phi$ satisfying
$\Theta(\phi \times \nu) =0$, which gives
$\phi \times \nu \cdot \Sigma (E \times \nu) =\phi \times \nu \cdot
\Sigma\Lambda ((H\times\nu)\times\nu)$.  We note that
$(\Sigma \Lambda)^\top$ is the identity on $Z^\perp$ and, thus
$(\Sigma \Lambda)^\top (\phi \times \nu)= \phi \times \nu$.  We obtain
$\phi \times \nu \cdot \Sigma (E \times \nu) =\phi \times \nu \cdot
((H\times\nu)\times\nu) = -\phi \times \nu \cdot H$, which is
\eqref{eq:BC-in-E-2}.

With these calculations, we have verified the last statement of Lemma
\ref{lem:AssumptionsLambda}.

\medskip Vice versa, we can motivate the weak formulation \eqref{eq:WeakForm:Maxwell-in-E} starting from the strong Maxwell system
\eqref{eq:Maxwell-strong-E-H:1}--\eqref{eq:Maxwell-strong-E-H:2} with
the boundary condition \eqref{eq:BC-in-E}: When we multiply
\eqref{eq:Maxwell-strong-E-H:1} with $\mu^{-1}$ and use, for arbitrary
$\phi \in H_\Theta(\curl, \Omega, \Gamma)$, the test-function
$\curl \phi$, we obtain
\begin{align}\label{eq:Derivation-WeakForm:1}
  \int\limits_\Omega \mu^{-1} \curl E \cdot \curl \phi
  = \int\limits_\Omega \left\{i \omega H \cdot \curl \phi
  + \mu^{-1} f_h\cdot \curl \phi\right\}\,.
\end{align}
We study the first term on the right-hand side. Integrating by parts
and using the identities \eqref{eq:Maxwell-strong-E-H:2} and
\eqref{eq:BC-in-E-2} we find
\begin{align}
	\begin{aligned}\label{eq:Derivation-WeakForm:2}
          \int\limits_\Omega i \omega H \cdot \curl \phi &= i \omega
          \int\limits_\Omega \curl H \cdot \phi + i \omega
          \int\limits_{\Gamma} H \cdot \nu \times \phi
          \\
          &= \omega^2 \int\limits_\Omega \eps E \cdot \phi +i \omega
          \int\limits_\Omega f_e \cdot \phi + i \omega
          \int\limits_{\Gamma} \Sigma(\nn{E}) \cdot \nn{\phi}\,.
	\end{aligned}
\end{align}
With this replacement in \eqref{eq:Derivation-WeakForm:1}, we find the
weak form \eqref{eq:WeakForm:Maxwell-in-E}. The identity
\eqref{eq:BC-in-E-1} restricts the solution space to
$H_\Theta(\curl, \Omega, \Gamma)$.

\subsection{Coercivity}\label{eq:Weak-Strong-BC-Equivalence}

Under certain assumptions on $\Lambda$, Lemma
\ref{lem:AssumptionsLambda} provides a suitable choice of $\Sigma$ and
$\Theta$. We have seen above that, with this choice, the weak solution
concept encodes \eqref{eq:Maxwell-strong-E-H:3}. It remains to show
that, for $\Lambda$ as in Lemma \ref{lem:AssumptionsLambda}, the maps
$\Sigma$ and $\Theta$ are well-defined and satisfy Assumption
\ref{ass:Coefficents}.

\begin{proof}[Proof of Lemma \ref{lem:AssumptionsLambda}]
  In the situation of the lemma with
  $Z = \ker(\Lambda(x))\subset \C^3$, the map
  $\Lambda(x)|_{Z^\perp} \colon Z^\perp\to Z^\perp$ is invertible with
  lower bound (independent of $x$). Since, additionally, $\Lambda$
  satisfies a uniform upper bound by the property
  $\Lambda\in L^\infty$, the inverse
  $(\Lambda(x)|_{Z^\perp})^{-1} \colon Z^\perp\to Z^\perp$ is uniformly
  bounded and coercive. In particular, $\Sigma$ is well defined and
  coercive on $Z^\perp$,
  \begin{align}
    \bar\zeta \cdot\Sigma(x)(\zeta)   \geq c_\Sigma \|\zeta\|^2
    \qquad \text{for all } \zeta \in Z^\perp
  \end{align}
  for some constant $c_\Sigma > 0$ that is independent of
  $x$. Decomposing an arbitrary vector $\xi\in \C^3$ as
  $\xi = z + \zeta$ with $z\in Z$ and $\zeta\in Z^\perp$, we can
  calculate, suppressing the point $x$,
  \begin{align*}
    (\Theta +\Sigma) (\xi) \cdot \bar\xi
    &= (\Theta +\Sigma) (z + \zeta) \cdot \bar\xi
    = (z + \Sigma(\zeta)) \cdot \bar\xi\\
    &= \|z\|^2 + \Sigma(\zeta) \cdot \bar\zeta
    \ge \|z\|^2 + c_\Sigma \|\zeta\|^2\,.
  \end{align*}
  This provides the coercivity \eqref{eq:Coerc:Sigma+Theta} of
  $\Theta +\Sigma$.
\end{proof}

The coercivity assumption \eqref{eq:Coerc:Sigma+Theta} is designed in
such a way that the boundary integral in the weak form
\eqref{eq:WeakForm:Maxwell-in-E} together with the function space
$H_\Theta(\curl, \Omega, \Gamma)$ provides full control over the
$\|\nn{E}\|_{L^2(\Gamma)}$-norm on the boundary
$\Gamma=\del\Omega$. This can be seen with the following calculation
for arbitrary $E \in H_\Theta(\curl, \Omega, \Gamma)$:
\begin{equation}
  \label{eq:EstBoundary}
  c_0 \|\nn{E}\|_{L^2(\Gamma)}^2
  \leq \int\limits_{\Gamma} (\Theta +\Sigma) (\nn{E}) \cdot \nn{\bar E}
  =\int\limits_{\Gamma} \Sigma (\nn{E}) \cdot \nn{\bar E}\,,
\end{equation}
where the equality uses $\Theta(\nn{E}) = 0$.

\subsection{Formulation in $H$ instead of $E$}

System \eqref{eq:Maxwell-strong-E-H} permits also a weak formulation
in terms of the magnetic field. Up to boundary regularity, it reads:
Find $H\in H(\curl, \Omega)$ such that
\begin{eqnarray}\nonumber
  &\int\limits_\Omega \left\{\eps^{-1} \curl H \cdot \curl \psi
    -\omega^2 \mu H \cdot \psi\right\}
    - i \omega \int\limits_{\Gamma} \Lambda \nnn{H} \cdot \nnn{\psi}
  \\\label{eq:WeakForm:Maxwell-in-H}
  &\qquad =
    \int\limits_\Omega \left\{- i \omega  f_h \cdot \psi
    + \varepsilon^{-1} f_e \cdot \curl \psi \right\}
    \text{ for all } \psi \,,
\end{eqnarray}
where test-functions $\psi$ are chosen in the same space as $H$.  The
underlying space is $H(\curl, \Omega)$, the additional requirement is
that the projection of $\nnn{\psi}$ onto $\operatorname{R}(\Lambda)$
is in the space $L^2(\Gamma)$.

For a solution $H\in H(\curl, \Omega, \Gamma)$ of
\eqref{eq:WeakForm:Maxwell-in-H}, we set
$E\coloneqq (i \omega \eps)^{-1} \left(-\curl H + f_e\right)$. The
formulation \eqref{eq:WeakForm:Maxwell-in-H} has the advantage that it
uses $\Lambda$. It does not require to formulate the problem with the
two auxiliary matrix functions $\Sigma$ and $\Theta$.

\smallskip Unfortunately, the weak formulation
\eqref{eq:WeakForm:Maxwell-in-H} has a major disadvantage concerning
compactness.  We recall that the literature provides compact
embeddings of $H(\curl, \Omega)\cap H(\div, \Omega)$ in $L^2(\Omega)$
when an $L^2(\Gamma)$-control of tangential or normal components of
the functions are available. For a singular map $\Lambda \neq 0$, we
have only control of one tangential component through the boundary
term since $\Lambda$ is not coercive on the entire tangent space.  At
the same time, we do not have full control of the normal component:
Inserting $\psi = \nabla \fhi$ in \eqref{eq:WeakForm:Maxwell-in-H},
assuming, for simplicity, $f_h =0$, we find
\begin{equation}
  -\int\limits_\Omega \omega^2\mu H \cdot \nabla \fhi
  = i\omega \int\limits_\Gamma \Lambda \nnn{H} \cdot  \nnn{\nabla\varphi}\,.
\end{equation}
The boundary term does not disappear for arbitrary
$\fhi \in H^1(\Omega)$ but only for $\fhi \in H^1(\Omega)$ such that
$\nnn{\nabla \fhi} \in \operatorname{R}(\Lambda)^\perp =
\ker(\Lambda)$. Thus, we obtain only a restricted information on the
normal component of $H$.  Consequently, one has a mixed control over
the tangential and normal components.  To the best knowledge of the
authors, no compactness result of the literature is applicable in this
setting.

\section{Helmholtz decomposition}\label{sec:Helmholtz}

Our aim is to prove Theorem \ref{thm:FredholmAlternative} with the
help of the compactness result of Lemma~\ref{lem:MaxwellCompactness}. This compactness result requires a control
of the divergence of $E$. We obtain this control in two steps: With a
Helmholtz decomposition of $L^2(\Omega, \C^3)$, we restrict the
analysis to divergence-free right-hand sides, see Lemma
\ref{lem:div-free-data}. With a Helmholtz decomposition of
$H_\Theta(\curl, \Omega,\Gamma)$, we restrict the set of solutions to
$\eps$-divergence-free functions.

We start by introducing the space $G$ of gradients. The set $G$ can be
understood as a subspace of $L^2(\Omega,\C^3)$, but also as a subspace
of $H(\curl,\Omega)$ since the rotation of gradients vanishes.  It is
even a subspace of $H_\Theta(\curl,\Omega, \Gamma)$ since the
tangential derivatives of an $H^1_0$-function vanish.  For a given
coefficient $\eps$, we define $D_\eps$ and $Y_\eps$ as spaces of
$\eps$-divergence-free functions. In the subsequent definition,
differential operators are understood in the sense of distributions.
\begin{align}
  \label{eq:space-G}
  G &\coloneqq \left\{ u\in L^2(\Omega)\mid
      \exists \psi\in H^1_0(\Omega) \colon
      u = \nabla\psi \right\} \,,
  \\
  \label{eq:space-Deps}
  D_\eps &\coloneqq \left\{ u\in L^2(\Omega, \C^3)\mid \div (\eps u) = 0 \right\} \,,
  \\
  \label{eq:space-Yeps}
  Y_\eps &\coloneqq H_\Theta(\curl \Omega, \Gamma) \cap D_\eps \,.
\end{align}
The choice is such that $D_\eps$ is the orthogonal complement of $G$
in the space $L^2(\Omega, \C^3)$ with the weighted scalar product
$\la u, v\ra_\eps = \int_\Omega \eps\, u\cdot \bar{v}$. Furthermore,
because of $\curl(\nabla\psi) = 0$ and
$\nu \times \nabla \psi|_{\Gamma} = 0$, the subspaces $Y_\eps$ and $G$
are also orthogonal with respect to the scalar product
$\la u, v\ra_X \coloneqq \int_\Omega \{ \eps\, u\cdot \bar v +
\mu^{-1} \curl u\cdot \curl\bar v \} + \int_{\Gamma} \{\nu \times
u\cdot \nu \times \bar v \}$. By construction, $Y_\eps$ is the
$\la ., .\ra_X$-orthogonal complement of $G$ in
$H_\Theta(\curl, \Omega, \Gamma)$. The definitions therefore imply
directly the following two Helmholtz decompositions.

\begin{lemma}[Helmholtz decomposition]\label{lem:Helmholtz-dec}
  The space $L^2(\Omega, \C^3)$ has the orthogonal decomposition
  $L^2(\Omega, \C^3) = D_\eps \oplus_\eps G$. In particular, an
  arbitrary element $u\in L^2(\Omega)$ can be written uniquely as
  $u = v + \nabla\psi$ with $v\in D_\eps$ and $\psi\in H^1_0(\Omega)$.

  The space $X \coloneqq H_\Theta(\curl,\Omega, \Gamma)$ has the
  orthogonal decomposition $X = Y_ \eps \oplus_X G$. In particular, an
  arbitrary element $u\in X$ can be written uniquely as
  $u = v + \nabla\psi$ with $v\in Y_\eps$ and $\psi\in H^1_0(\Omega)$.
\end{lemma}

We use the Helmholtz decomposition of $L^2(\Omega, \C^3)$ in order to
replace the data $f_e$ with divergence-free data. This allows us to
control the divergence of the unknown $E$.

\begin{lemma}[Reduction to divergence-free
  data]\label{lem:div-free-data}
  Let $f_e, f_h \in L^2(\Omega, \C^3)$ be given and let $\eps$ satisfy
  Assumption \ref{ass:Coefficents}.  Using Lemma
  \ref{lem:Helmholtz-dec}, we find $h\in D_\eps$ and
  $\chi \in H^1_0(\Omega)$ such that
  $(i\omega \eps)^{-1}f_e = h + \nabla \chi$. We use
  $\tilde{f}_e \coloneqq (i\omega \eps) h$ with
  $\div(\tilde{f}_e) = 0$. Then,
  $E \in H_\Theta(\curl,\Omega, \Gamma)$ is a solution for
  \eqref{eq:WeakForm:Maxwell-in-E} if and only if
  $\tilde{E} = E - \nabla \chi \in H_\Theta(\curl,\Omega, \Gamma)$
  satisfies
  \begin{align}
    \label{eq:Maxwell-in-tilde-E}
    \begin{aligned}
      &\int\limits_\Omega \left\{\mu^{-1} \curl \tilde{E} \cdot \curl \phi
        -\omega^2 \eps \tilde{E} \cdot \phi\right\}
        - i \omega \int\limits_{\Gamma} \Sigma (\nn{\tilde{E}}) \cdot \nn{\phi}
        \\
      &\qquad\qquad = \int\limits_\Omega \left\{i \omega \tilde{f}_e \cdot \phi
        + \mu^{-1} f_h \cdot \curl \phi \right\}
        \quad \forall \phi \in H_\Theta(\curl,\Omega, \Gamma)\,.
    \end{aligned}
  \end{align}
  Moreover, solutions $\tilde{E}$ are $\eps$-divergence-free in the
  sense that $\div(\eps\tilde{E})=0$.
\end{lemma}

\begin{proof}
  When $E$ is a solution, elementary substitutions show that
  $\tilde E$ is a solution of \eqref{eq:Maxwell-in-tilde-E}. Indeed,
  the curl of a gradient vanishes, the second terms on both sides are
  modified in the same way, the boundary integral is unchanged since
  $\chi$ vanishes on $\Gamma$ and, thus, tangential components of
  $\nabla\chi$ vanish along the boundary.

  The opposite implication is obtained with the same calculation.
  
  In order to obtain $\div(\eps\tilde{E})=0$, it is sufficient to use
  a gradient $\phi = \nabla\fhi$ for $\fhi \in H^1_0(\Omega)$ in
  \eqref{eq:Maxwell-in-tilde-E}.
\end{proof}

Lemma \ref{lem:div-free-data} allows us to consider only right-hand
sides $f_e$ with $\div(f_e) =0$ in the following.  By doing so, we can
also restrict the solution space (and, accordingly, the space of
test-functions) in \eqref{eq:WeakForm:Maxwell-in-E} to
$\eps$-divergence-free functions.  For these functions, we write $E$
and $\phi$, dropping the tilde.

\begin{lemma}[Equivalent formulation in $Y_\eps$]
  \label{lem:Maxwell-Y}
  Let $\Omega\subset\R^3$ be a bounded Lipschitz domain, $\omega>0$,
  $\eps, \mu$, $\Sigma, \Theta$ as in Assumption
  \ref{ass:Coefficents}. Let $f_e, f_h \in L^2(\Omega, \C^3)$ be
  right-hand sides with $\div(f_e) = 0$. In this situation, the weak
  problem formulated in \eqref{eq:WeakForm:Maxwell-in-E} is equivalent
  to the following problem: Find $E \in Y_\eps$ such that the equation
  in \eqref{eq:WeakForm:Maxwell-in-E} holds for all test-functions
  $\phi \in Y_\eps$.
\end{lemma}

\begin{proof}
  Let $E$ be a solution of the new problem, i.e., $E\in Y_\eps$ and
  \eqref{eq:WeakForm:Maxwell-in-E} holds for test functions
  $\phi \in Y_\eps$. Given an arbitrary test-function
  $\phi \in X = H_\Theta(\curl, \Omega, \Gamma)$, we write
  $\phi = \fhi + \nabla\psi$ with $\fhi\in Y_\eps$ and
  $\psi\in H^1_0(\Omega)$ as outlined in Lemma \ref{lem:Helmholtz-dec}.  Using that \eqref{eq:WeakForm:Maxwell-in-E} is
  linear in $\phi$, we can treat the contributions
  separately. Inserting $\nabla\psi$ in
  \eqref{eq:WeakForm:Maxwell-in-E}, all terms vanish.  Inserting
  $\fhi$, the equality holds since $E$ is a solution of the
  $Y_\eps$-problem.  This shows that \eqref{eq:WeakForm:Maxwell-in-E}
  is satisfied for arbitrary test-functions
  $\phi \in H_\Theta(\curl,\Omega, \Gamma)$.

  When $E$ is a solution of the original weak form, then
  \eqref{eq:WeakForm:Maxwell-in-E} holds, in particular, for
  test-functions $\phi \in Y_\eps$. The fact that $E$ is indeed an
  element of $Y_\eps$ was observed in Lemma \ref{lem:div-free-data}.
\end{proof}

\section{Verification of the Fredholm alternative}
\label{sec:Fredholm}

With the above considerations, the Maxwell system has a symmetric weak
formulation in the space $Y_\eps$ (we recall that now only
divergence-free right-hand sides $f_e$ are considered). Using the
compactness result of Lemma \ref{lem:MaxwellCompactness} below, some
functional analysis provides the Fredholm alternative of Theorem
\ref{cor:existence}.

\begin{proof}[Proof of Theorem \ref{cor:existence}]
  We recall that $Y_\eps \subset H_\Theta(\curl \Omega, \Gamma)$
  denotes the subspace of $\eps$-divergence-free functions. We define
  two sesquilinear forms $a, b \colon Y_\eps \times Y_\eps \to \C$ and
  an anti-linear right-hand side $f \colon Y_\eps \to \C$ by setting,
  for every $u, \phi\in Y_\eps$,
  \begin{align*}
    &a(u, \phi) \coloneqq
      \int\limits_\Omega \left\{ u\cdot \bar{\phi}
      + \mu^{-1}\curl u \cdot \curl \bar{\phi} \right\} - i \omega
      \int\limits_{\Gamma} \left\{\Sigma(\nn{u})\cdot \nn{\bar{\phi}} \right\}\,, 
    \\
    &b(u, \phi) \coloneqq
      \int\limits_\Omega \left\{ u\cdot \bar{\phi}
      + \omega^2 \eps\, u \cdot \bar{\phi}\right\}\,, \quad
      f(\phi) \coloneqq
      \int\limits_\Omega \left\{i \omega f_e \cdot \bar{\phi}
      + \mu^{-1} f_h \cdot  \curl \bar{\phi} \right\} \,.
  \end{align*}
  By Lemma~\ref{lem:Maxwell-Y}, the weak formulation of the Maxwell
  problem is equivalent to: Find $E\in Y_\eps$ such that
  \begin{equation}
    \label{eq:bilinear-equation}
    a(E, \phi) - b(E, \phi) = f(\phi) \qquad \forall\ \phi \in Y_\eps\,.
  \end{equation}
  
  The sesquilinear form $a$ defines a map $A \colon Y_\eps\to Y_\eps'$
  from $Y_\eps$ into the (anti-)dual space $Y_\eps'$ with the
  definition $Au \coloneqq a(u, \cdot)$. By definition of the scalar
  product in $Y_\eps \subset X = H_\Theta(\curl,\Omega, \Gamma)$ and
  the estimate \eqref{eq:EstBoundary}, the form $a$ is coercive on
  $Y_\eps$.  The Lemma of Lax--Milgram implies that
  $A \colon Y_\eps\to Y_\eps'$ is invertible.
	
  We now exploit that the embedding
  $\iota \colon Y_\eps \to L^2(\Omega,\C^3)$ is compact, see
  Lemma~\ref{lem:MaxwellCompactness}. The multiplication map
  $B \colon u \mapsto (1 + \omega^2 \eps) \,u$ corresponding to $b$ is
  linear and bounded as a map
  $B \colon L^2(\Omega,\C^3) \to L^2(\Omega,\C^3)$. We denote the
  concatenation with an embedding into $Y_\eps'$ with the same letter
  and write $B \colon L^2(\Omega,\C^3) \to Y_\eps'$.
	
  The field $E\in Y_\eps$ solves \eqref{eq:bilinear-equation} if and
  only if
  \begin{align*}
    A\, E - (B\circ\iota)\, E = f \qquad \text{ in } Y_\eps'\,.
  \end{align*}
  Applying $A^{-1}$, we find the equivalent relation
  \begin{equation*}
    (\operatorname{id} - A^{-1} \circ B\circ\iota) E = A^{-1} f
    \qquad \text{ in } Y_\eps\,.
  \end{equation*}
  The operator $A^{-1} \circ B\circ \iota$ is compact, since $\iota$
  is compact and the other operators are continuous. Standard
  functional analysis results imply that the operator
  $F \coloneqq \operatorname{id} - A^{-1} \circ B\circ\iota$ is a
  Fredholm operator of index zero, see, e.g., \cite[Theorem
  11.8]{Alt16}. By their definition, such operators satisfy the
  Fredholm alternative: The kernel is trivial if and only if the
  operator is surjective.
\end{proof}

\section{Compactness property}\label {sec.compactness}

In the previous section, we derived the Fredholm
alternative from a compact embedding: We need that the solution space
is compactly embedded in $L^2(\Omega, \C^3)$. In our application, we
considered functions with a vanishing $\eps$-divergence, but this is
actually not needed for the compactness. The compactness only needs
that the $\eps$-divergence is controlled in $L^2$ (just as we demand
the control of the curl).  We use the following spaces:
\begin{eqnarray}\nonumber
  H(\div \eps,\Omega) \coloneqq \left\{u \in L^2(\Omega,\C^3)
  \ \middle|\ 
  \exists f \in L^2(\Omega, \C)\colon\phantom{\int}\right.
  \qquad\qquad \\
  \left.
  \int_\Omega f \phi = \int_\Omega \eps u \cdot \nabla \phi \
  \forall \phi \in C^\infty_c(\Omega) \right\}\,,
  \label{eq:H-div}
\end{eqnarray}
and
\begin{eqnarray}\nonumber
  H(\div \eps, \Omega, \Gamma) \coloneqq \left\{u \in H(\div \eps,\Omega)
  \ \middle|\ 
  \exists g \in L^2(\Gamma, \C) \colon\phantom{\int}\right.
  \qquad\qquad\qquad
  \\
  \left.
  \int_\Omega \left\{\div (\eps u) \phi + \eps u \cdot \nabla \phi \right\}
  = \int_\Gamma g \phi \quad \forall \phi \in H^1(\Omega)\right\}\,.
  \label{eq:HGL2-div}
\end{eqnarray}
Norms and scalar products are defined in the natural way, using
$L^2$-norms of all the given quantities.  For functions $u$ in the
second space, the normal trace of $u$ is defined as
$\nu \cdot (\eps u)|_{\Gamma} \coloneqq g$.

The following two compactness results are well known. They both
require an $L^2(\Omega)$-control over curl and divergence and an
$L^2(\Gamma)$-control of either the normal or tangential boundary
values. In our main result, we have used the compact embedding \eqref{eq:compact-claim1}. The space $Y_\eps$ is a subspace of the left-hand
side; for functions $u\in Y_\eps$ holds $\div(\eps u) = 0$, which
implies, in particular, that $\div(\eps u)$ is an
$L^2(\Omega)$-function. Additionally, for functions $u\in Y_\eps$,
there holds $\Theta(\nn{u}) = 0$, but we exploit only the
$L^2(\Gamma)$-control of $\nn{u}$.

In the result \eqref{eq:compact-claim2}, the normal trace of
functions is controlled in $L^2(\Gamma)$.

\begin{lemma}[Mawell compactness
  theorem]\label{lem:MaxwellCompactness}
  Let $\Omega\subset \R^3$ be a Lipschitz domain and $\eps, \mu$ as in
  Assumption \ref{ass:Coefficents}. Then, the following embeddings are
  compact:
  \begin{align}
    \label{eq:compact-claim1}
    H(\curl, \Omega, \Gamma) \cap H(\div \eps,\Omega)
    \cptemb L^2(\Omega, \C^3)\,,
    \\
    \label{eq:compact-claim2}
    H(\curl, \Omega) \cap H(\div \mu,\Omega, \Gamma)
    \cptemb L^2(\Omega, \C^3)\,.
  \end{align}
\end{lemma}

The second embedding is used when the Maxwell system is formulated in
$H$; for this reason, the coefficient $\mu$ (instead of $\eps$) is
typically used in this formulation. In order to highlight the
controlled quantities, one may write
$H(\curl, \div \eps, \Omega, \nn{})$ for the space on the left-hand
side of \eqref{eq:compact-claim1} and
$H(\curl, \div \mu, \Omega, \nu \cdot |_{\Gamma})$ for the space on
the left-hand side of \eqref{eq:compact-claim2}.

We note that the compactness has some relations with the div--curl
lemma, sometimes called compensated compactness. Knowledge on the curl
and on the divergence of a function somehow controls all
derivatives. Loosely speaking, this property is already suggested by
the relation $\Delta = -\curl\curl + \nabla \div$. We will actually
exploit this relation in our proof.

We sketch a proof for the first embedding, which is used in this
article. The proof for the second embedding is identical in most
steps, the only difference is that a Neumann problem instead of a
Dirichlet problem must be analyzed in Step 3. We note that the proof
for a vanishing normal component with the same methods can be found in
\cite{Existence-H-field}.

\begin{proof}
  In Steps 1--3 of this proof, we assume that $\Omega$ is simply
  connected and that the coefficient is $\eps\equiv 1$.  In Step 4 we
  treat general coefficients $\eps$ and in Step 5 we show how the
  assumption of simple connectedness is removed.
  
  In order to show compactness, we consider a bounded sequence $u_j$
  in the space $H(\curl, \Omega, \Gamma) \cap H(\div \eps,\Omega)$.
  We recall that this implies that $u_j$, $f_j \coloneqq \curl u_j$,
  $g_j \coloneqq \div (\eps u_j)$ and the tangential boundary values of $u_j$
  are bounded sequences in $L^2$-spaces.

  The set $\Omega$ is bounded, we can therefore choose a radius $R>0$
  such that $\Omega$ is compactly contained in the open ball with this
  radius,
  $\overline{\Omega} \subset B_R \coloneqq B_R(0) \subset \R^3$.
  
  \smallskip {\em Step 1: Extension of $f_j$ with a gradient.}  We
  consider the solutions $\psi_j$ of the following Laplace problem:
  $\Delta\psi_j = 0$ in $B_R\setminus \overline{\Omega}$ with
  $\psi_j = 0$ on $\del B_R$ and $\del_\nu\psi_j = \nu\cdot f_j$ on
  $\del\Omega$. We define an extension of $f_j$ by setting
  \begin{equation*}
    \tilde{f_j} \coloneqq
    \begin{cases}
      f_j\qquad &\text{in } \Omega\,,\\
      \nabla\psi_j\qquad &\text{in } B_R\setminus \Omega\,.
    \end{cases}
  \end{equation*}
  The construction ensures that this extended function has a vanishing
  divergence. We note that the function $f_j$ is a curl and has
  therefore a vanishing divergence; this fact actually allows us to
  formulate the boundary condition in the Neumann problem even though
  $f_j$ is only an $L^2$-function. A detailed verification of this
  fact can be found in \cite{Existence-H-field}.
  
  \smallskip {\em Step 2: Vector potential $\tilde v_j$ for
    $\tilde{f_j}$.} As a function with vanishing divergence on the
  simply connected set $B_R$, the function $\tilde{f_j}$ possesses a
  divergence-free vector potential $\tilde v_j\in L^2(B_R, \C^3)$ with
  $\curl \tilde v_j = \tilde{f_j}$ in $B_R$. One way to show this
  standard result is the following: In a first step, a function
  $\tilde w_j$ with $\curl \tilde w_j = \tilde{f_j}$ is constructed
  with the help of path integrals of $\tilde{f_j}$. When an average
  over a set of starting points of the path integrals is taken, the
  function $\tilde w_j$ is of class $L^2(B_R, \C^3)$ with norm bounded
  by the norm of $\tilde f_j$. In a second step, we subtract the
  gradient $\nabla\tilde\xi_j$ of a function $\tilde\xi_j$ with
  $\Delta\tilde\xi_j = \div(\tilde w_j)$. We obtain that
  $\tilde v_j = \tilde w_j - \nabla\tilde\xi_j$ has the same rotation
  as $\tilde w_j$ (namely $\tilde{f_j}$) and satisfies
  $\div \tilde v_j = 0$.

  Let us discuss the regularity of $\tilde v_j$. Because of
  $\Delta = -\curl\curl + \nabla \div$, the function
  $\tilde v_j \in L^2(B_R, \C^3)$ satisfies, in the sense of
  distributions,
  $\Delta \tilde v_j = -\curl \tilde f_j \in H^{-1}(B_R)$. The
  sequence $\tilde v_j$ is therefore locally (in $B_R$) of class $H^1$
  (Caccioppoli's inequality).  This implies that
  $v_j \coloneqq \tilde v_j|_\Omega$ is bounded in $H^1(\Omega)$. In
  particular, it possesses a subsequence that converges strongly in
  $L^2(\Omega, \C^3)$. We consider only this subsequence in the
  following.
  
  \smallskip {\em Step 3: Scalar potential for $u_j - v_j$.}  We
  consider now only functions on the simply connected set $\Omega$.
  By construction, the difference $u_j - v_j$ has a vanishing curl, it
  can therefore be written, for some scalar function
  $\phi_j \in H^1(\Omega)$, as a gradient: $u_j - v_j = \nabla\phi_j$.

  Let us analyze the regularity properties of $\phi_j$ along $\Gamma$.
  We have bounds for $\nn{u_j} \in L^2(\partial \Omega)$ by the
  assumptions on $u_j$ and we have bounds for
  $\nn{v_j} \in L^2(\partial \Omega)$ because of the boundedness of
  $v_j\in H^1(\Omega)$ and classical trace theorems. We therefore have
  $\nn{\nabla \phi_j} = \nn{u_j} - \nn{v_j} \in L^2(\partial \Omega)$
  bounded and conclude the boundedness of
  $\phi_j \in H^1(\partial \Omega)$.

  It remains to analyze the Dirichlet problem
  \begin{equation*}
    \Delta\phi = g\quad \text{ in } \Omega\,,\qquad
    \phi = h\quad \text{ on } \Gamma\,,
  \end{equation*}
  for given $g$ and $h$.  By classical elliptic theory, the solution
  operator to this problem is a bounded linear operator
  $\cT \colon H^{-1}(\Omega) \times H^{1/2}(\Gamma) \ni (g, h) \mapsto
  \phi \in H^1(\Omega)$.  We have the compact embeddings
  $L^2(\Omega) \cptemb H^{-1}(\Omega)$ and
  $H^1(\Gamma)\cptemb H^{1/2}(\Gamma)$. They imply that the operator
  $\cT$ on these better function spaces is compact,
  \begin{equation*}
    \tilde{\cT} \colon L^2(\Omega) \times H^{1}(\Gamma) \ni (g, h) \mapsto
    \phi \in H^1(\Omega)\quad \text{compact}\,.
  \end{equation*}
  As a solution of the Dirichlet problem,
  $\phi_j = \tilde{\cT}(g_j, \phi_j|_\Gamma)$, the sequence $\phi_j$
  has a subsequence that converges in $H^1(\Omega)$.  Along this
  subsequence, both $\nabla\phi_j$ and $v_j$ are converging strongly
  in $L^2(\Omega, \C^3)$, hence also $u_j = v_j + \nabla\phi_j$ is
  converging in this space.  This concludes the compactness proof for
  simply connected domains and $\eps\equiv 1$.

  \smallskip {\em Step 4: General coefficients $\eps$.}  We now
  consider a general coercive coefficient
  $\eps\in L^\infty(\Omega, \C^{3 \times 3})$. In order to show
  compactness, we consider once more a bounded sequence
  $u_j\in H(\curl, \Omega, \Gamma) \cap H(\div \eps,\Omega)$.
  
  We claim that it is sufficient to consider the case
  $\div(\eps u_j) = 0$. Indeed, in the general case, we consider the
  sequence $\hat u_j \coloneqq u_j - \nabla\xi_j$ where
  $\xi_j\in H^1_0(\Omega)$ solves
  $\div(\eps\nabla\xi_j) = \div(\eps u_j)$. For a subsequence,
  $\div(\eps u_j)$ is converging weakly in $L^2(\Omega)$ and hence
  strongly in $H^{-1}(\Omega)$, which implies that $\xi_j$ is strongly
  converging in $H^1(\Omega)$ and $\nabla\xi_j$ strongly in
  $L^2(\Omega)$. It is therefore sufficient to prove the strong
  convergence of $\hat u_j$ in $L^2(\Omega)$. This justifies the
  claim.
  
  We use the Helmholtz decomposition
  $L^2(\Omega,\C^3) = D_\eps \oplus_\eps G$ with $D_\eps$ and $G$
  defined in \eqref{eq:space-Deps} and \eqref{eq:space-G}.  We
  discussed this decomposition in Lemma \ref{lem:Helmholtz-dec},
  where we also noted that
  $\tilde D_\eps \coloneqq H(\curl,\Omega, \Gamma)\cap D_\eps$ allows us to
  introduce the decomposition
  $H(\curl,\Omega, \Gamma) = \tilde D_\eps \oplus G$.  The corresponding
  projections are bounded.

  We decompose $u_j$ with respect to the decomposition that
  corresponds to $\eps = \id$, that is, using
  $H(\curl,\Omega, \Gamma) = \tilde D_\id\oplus G$:
  \begin{equation*}
    u_j\ =\ v_j\ +\ \nabla \psi_j\quad\text{with $v_j\in \tilde D_\id$ and }
    \psi_j\in H^1_0(\Omega,\C)\,.
  \end{equation*}
  Since $u_j$ is bounded in $H(\curl,\Omega, \Gamma)$ and since projections
  are bounded, the sequence $v_j\in \tilde D_\id$ is bounded in
  $H(\curl,\Omega, \Gamma)$. Since it has a vanishing divergence, Steps 1--3
  yield that there exists a subsequence, again denoted by $v_j$, which
  converges in $L^2(\Omega,\C^3)$. With this knowledge, we now read
  the previous decomposition in the form
  \begin{equation*}
    v_j\ =\ u_j\ -\ \nabla \psi_j\,,
  \end{equation*}
  and note that this is a decomposition of $v_j$ in
  $L^2(\Omega,\C^3) = D_\eps\ \oplus\ G$. Since $v_j$ converges in
  $L^2(\Omega,\C^3)$ and since the projection onto $D_\eps$ is bounded
  in the space $L^2(\Omega,\C^3)$, we conclude that $u_j$ converges in
  $L^2(\Omega,\C^3)$.  This concludes the proof.

  \smallskip {\em Step 5: Removing the assumption of simple
    connectedness.} It remains to consider an arbitrary bounded
  Lipschitz domain $\Omega$. Let again $u_j$ be a bounded sequence in
  the spaces on the left-hand side. We choose a finite family of
  simply connected Lipschitz subdomains $\Omega_k\subset \Omega$,
  $k=1,...,K$, such that $\Omega$ is covered,
  $\Omega\subset \bigcup_{k=1}^K \Omega_k$. We choose a subordinate
  family of smooth cut-off functions $\eta_k$ with
  $\bigcup_{k=1}^K \eta_k = 1$ on $\Omega$. Steps 1--4 can be applied
  successively in the subdomains $\Omega_k$ to the sequences
  $u_j \eta_k$ to find $L^2(\Omega)$-convergent subsequences. For the
  corresponding subsequence, also $u_j = \sum_{k=1}^K u_j \eta_k$ is
  then $L^2(\Omega)$-convergent.
\end{proof}

\section{Limiting absorption principle}\label{sec:LimitingAbs}

In this section, we prove Corollary \ref{cor:existence} with a
limiting absorption principle. In the limiting absorption principle,
one introduces a damping term in the equation and considers the limit
of a vanishing damping (or absorption). Here, we use a small real
number $\delta>0$ and replace the pre-factor $\omega^2 \eps$ in
equation \eqref{eq:WeakForm:Maxwell-in-E} by
$\omega^2 \eps + i \delta$. We will verify that this equation has a
solution $E_\delta$, we will find a limit
$E = \lim_{\delta\to 0} E_\delta$, and we show that $E$ is a solution
to \eqref{eq:WeakForm:Maxwell-in-E}.

The problem with absorption is: Find
$E_\delta \in H_\Theta(\curl, \Omega, \Gamma)$ such that
\begin{eqnarray}
  \nonumber
  \int\limits_\Omega \left\{\mu^{-1} \curl E_\delta \cdot \curl \phi
  -(\omega^2 \eps + i\delta) E_\delta \cdot \phi\right\}
  - i \omega \int\limits_{\Gamma} \Sigma(\nn{E_\delta}) \cdot \nn{\phi}
  \\
  = \int\limits_\Omega \left\{i \omega f_e \cdot \phi
  + \mu^{-1} f_h \cdot \curl \phi \right\}
  \quad \forall\ \phi\in H_\Theta(\curl, \Omega, \Gamma)\,.
  \label{eq:WeakForm:Maxwell-in-E-delta}
\end{eqnarray}

\begin{proof}[Proof of Corollary \ref{cor:existence} with limiting
  absorption]
  Lemma \ref{lem:div-free-data} guarantees that we can assume, without
  loss of generality, that the data satisfy $\div(f_e) = 0$.
  Lemma~\ref{lem:Existence-delta} below provides a unique solution
  $E_\delta$ of \eqref{eq:WeakForm:Maxwell-in-E-delta}.
  Lemma~\ref{lem:Boundedness-delta} below shows that the sequence
  $E_\delta$ is bounded in $H_\Theta(\curl, \Omega, \Gamma)$ (in the
  setting of Corollary \ref{cor:existence}, where it is assumed that
  \eqref{eq:WeakForm:Maxwell-in-E} has only the trivial solution for
  $f_h = f_e = 0$). Since $H_\Theta(\curl, \Omega, \Gamma)$ is
  reflexive, there exists a subsequence of $E_\delta$ that converges
  weakly to some limit $E \in H_\Theta(\curl, \Omega, \Gamma)$. The
  weak convergence allows us to take the limit $\delta \to 0$ in
  \eqref{eq:WeakForm:Maxwell-in-E-delta}. We obtain that $E$ solves
  \eqref{eq:WeakForm:Maxwell-in-E}.
\end{proof}

\begin{lemma}[Existence of a solution for the problem with
  absorption]\label{lem:Existence-delta}
  Let $\Omega\subset\R^3$ be a bounded Lipschitz domain and $\omega>0$
  and $\eps, \mu$ and $\Sigma, \Theta$ satisfy
  Assumption~\ref{ass:Coefficents}. Let
  $f_e, f_h \in L^2(\Omega, \C^3)$ be right-hand sides with
  $\div(f_e) = 0$.  Then, there exists $\delta_0>0$ such that, for
  every $\delta \in (0,\delta_0)$, equation
  \eqref{eq:WeakForm:Maxwell-in-E-delta} has a unique weak solution
  $E_\delta \in H_\Theta(\curl, \Omega, \Gamma)$.
\end{lemma}

\begin{proof}
  We define a sesquilinear form $a_\delta$ on
  $H_\Theta(\curl, \Omega, \Gamma)$ by setting, for
  $u, \varphi \in H_\Theta(\curl, \Omega, \Gamma)$,
  \begin{equation}
    \label {eq:a-delta-lim-abs}
    a_\delta (u, \varphi) \coloneqq
    \int\limits_\Omega \left\{\mu^{-1}\curl u \cdot \curl \bar{\varphi}
      - (\omega^2 \eps +i \delta) u \cdot \bar{\varphi} \right\}
    - i \omega \int\limits_{\Gamma} \Sigma(\nn{u}) \cdot \nn{\bar{\varphi}} \,.
  \end{equation}
  The form $a_\delta$ allows us to rewrite
  \eqref{eq:WeakForm:Maxwell-in-E-delta} as
  \begin{equation}\label{eq:Maxwell-delta-absorption}
    a_\delta(E_\delta, \phi)
    = \int\limits_\Omega \left\{i \omega f_e \cdot \bar{\phi}
      + \mu^{-1} f_h \cdot \curl \bar{\phi}\right\}
    \qquad \forall \, \phi \in H_\Theta(\curl, \Omega, \Gamma)\,.
  \end{equation}
  We calculate with the coercivity lower bound \eqref{eq:EstBoundary},
  for arbitrary $u \in H_\Theta(\curl, \Omega, \Gamma)$,
  \begin{align*}
    &|\Im a_\delta(u,u)| \geq \delta \| u \|_{L^2(\Omega)}^2
      + c_0 \omega \| \nn{u} \|_{L^2(\Gamma)}^2\,,\\
    &\Re a_\delta(u,u) \geq c_0
      \| \curl u \|_{L^2(\Omega)}^2
      - \omega^2 {
      \|\eps\|_{L^\infty(\Omega)}}\|u \|_{L^2(\Omega)}^2 \,.
  \end{align*}
  
  We observe the following fact in $\R^2 \equiv \C$: For every vector
  $z = (z_1,z_2) \in \R^2$ and every $s \in [0,1]$, there holds
  $|z| \ge \max\{ |z_1|, |z_2|\} \ge (1-s) |z_1| + s |z_2|$. This
  inequality allows us to calculate, with $s = \delta^2$,
  \begin{align*}
    &|a_\delta(u,u)|
      \geq (1- \delta^2)|\Im a_\delta(u,u)| + \delta^2 \Re a_\delta (u,u) \\
    &\quad \geq (1-\delta^2)\left( \delta \| u \|_{L^2(\Omega)}^2
      + c_0\omega\, \| \nn{u} \|_{L^2(\Gamma)}^2\right)
    \\
    &\qquad+ \delta^2 \left(c_0 \| \curl u \|_{L^2(\Omega)}^2
      - \omega^2 \|\eps\|_{L^\infty(\Omega)} \| u \|_{L^2(\Omega)}^2 \right)\,.
  \end{align*}
  Choosing $\delta_0>0$ small, we achieve
  $(1-\delta^2) \delta \ge 2 \delta^2\omega^2
  {\|\eps\|_{L^\infty(\Omega)}}$ for all $0 < \delta < \delta_0$; for
  these values of $\delta$, the form $a_\delta$ is coercive. Problem
  \eqref{eq:Maxwell-delta-absorption} for $E_\delta$ can therefore be
  solved with the Lemma of Lax--Milgram.
\end{proof}

\begin{lemma}[Boundedness of solutions to the problem with
  absorption]\label{lem:Boundedness-delta}
  Let the assumptions of Lemma \ref{lem:Existence-delta} be
  satisfied. For a sequence $\delta\to 0$, let
  $E_\delta \in H_\Theta(\curl,\Omega, \Gamma)$ be the corresponding
  sequence of solutions of \eqref{eq:WeakForm:Maxwell-in-E-delta}. We
  assume that relation \eqref{eq:WeakForm:Maxwell-in-E} with data
  $f_h = f_e = 0$ has only the trivial solution $E = 0$. Then, the
  sequence $E_\delta$ is bounded in $H_\Theta(\curl,\Omega,\Gamma)$.
\end{lemma}

\begin{proof}
  {\em Step 1: Preparation.} For a contradiction argument, we assume
  that there exists a subsequence $E_\delta$ such that
  $\|E_\delta\|_{H_\theta(\curl, \Omega, \Gamma)} \to \infty$.  The
  subsequent calculation uses that $c_0$ is a coercivity constant for
  the matrix function $\mu^{-1}$ and the fact that
  $|z_1 + i z_2| = \sqrt{|z_1|^2 + |z_2|^2} \ge \frac12 (|z_1| +
  |z_2|)$ holds for real numbers $z_1$ and $z_2$.  Using $E_\delta$ as
  a test-function in \eqref{eq:WeakForm:Maxwell-in-E-delta}, we find
  \begin{align*}
    &\frac{1}{2}\left(c_0 \|\curl E_\delta\|_{L^2(\Omega)}^2
      + \omega c_0 \|\nn{E_\delta}\|_{L^2(\Gamma)}^2\right)
    \displaybreak[2]\\
    &\qquad \leq
      \left|\int\limits_\Omega \mu^{-1}\curl E_\delta \cdot \curl \bar{E}_\delta 
      - i \omega \int\limits_{\Gamma} \Sigma(\nn{E_\delta}) \cdot \nn{\bar{E}_\delta}
      \right|
    \displaybreak[2]\\
    &\qquad =
      \left|\int\limits_\Omega (\omega^2 \eps + i \delta) E_\delta \cdot \bar{E}_\delta
      +
      \int\limits_\Omega \left\{ i \omega f_e \cdot \bar{E}_\delta
      + \mu^{-1} f_h \cdot  \curl \bar{E}_\delta \right\}\right|\,.
  \end{align*}
  For an arbitrarily small number $\lambda>0$, we can continue this
  calculation with Young's inequality to find
  \begin{eqnarray*}
    \frac{1}{2}\left(c_0 \|\curl E_\delta\|_{L^2(\Omega)}^2
    + \omega c_0 \|\nn{E_\delta}\|_{L^2(\Gamma)}^2\right)
    \\
    \leq C \|E_\delta\|_{L^2(\Omega)}^2 + C_\lambda + \lambda \|\curl E_\delta\|_{L^2(\Omega)}^2
    \,,
  \end{eqnarray*}
  for some $C$ depending on $\omega$, $f_e$ and $\eps$, and
  $C_\lambda$ depending on $\mu$, $f_h$ and $\lambda$. Choosing
  $\lambda = c_0/4$ and subtracting the term
  $\lambda \|\curl E_\delta\|_{L^2(\Omega)}^2$ on both sides, we find,
  for some constant $C$, the inequality
  $\|\curl E_\delta\|_{L^2(\Omega)}^2 +
  \|\nn{E_\delta}\|_{L^2(\Gamma)}^2 \le C\, (1 +
  \|E_\delta\|_{L^2(\Omega)}^2)$.
	
  In particular, we can conclude that our assumption
  $\|E_\delta\|_{H_\Theta(\curl, \Omega, \Gamma)} \to \infty$ implies
  the divergence of the $L^2$-norm,
  $\|E_\delta\|_{L^2(\Omega)} \to \infty$.

  \smallskip {\em Step 2: Normalization.} We normalize the sequence
  $E_\delta$ and consider the new sequence
  $\tilde{E}_\delta \coloneqq E_\delta / \|E_\delta\|_{L^2(\Omega,
    \C^3)}$. The normalized sequence satisfies
  $\|\tilde{E}_\delta\|_{L^2(\Omega, \C^3)} = 1$ and, by Step 1, that
  $\|\curl \tilde E_\delta\|_{L^2(\Omega)}^2 +
  \|\nn{E_\delta}\|_{L^2(\Gamma)}^2$ is bounded.  Since
  $H_\Theta(\curl, \Omega, \Gamma)$ is reflexive, there exists a
  subsequence of $\tilde{E}_\delta$ that converges weakly to some
  limit $\tilde{E} \in H_\Theta(\curl, \Omega, \Gamma)$. In
  particular, we have: $\tilde{E}_\delta\weakto \tilde{E}$ in
  $L^2(\Omega)$, $\curl \tilde{E}_\delta\weakto \curl \tilde{E}$ in
  $L^2(\Omega)$ and $\nn{\tilde{E}_\delta}\weakto \nn{\tilde{E}}$ in
  $L^2(\Gamma)$.

  The function $\tilde{E}_\delta$ solves
  \eqref{eq:WeakForm:Maxwell-in-E-delta} for source terms
  $\tilde f_e = f_e / \|E_\delta\|_{L^2(\Omega)}$ and
  $\tilde f_h = f_h / \|E_\delta\|_{L^2(\Omega)}$.  The weak
  convergence of $\tilde{E}_\delta$ allows us to perform the limit
  $\delta \to 0$ in this relation. We obtain that the limit
  $\tilde{E}$ solves \eqref{eq:WeakForm:Maxwell-in-E} for
  $f_h= f_e =0$. Our assumption was that there is no non-trivial
  solution to the homogeneous problem; this implies $\tilde{E} = 0$.

  \smallskip {\em Step 3: Strong convergence.} In this step we show
  the strong convergence of $\tilde{E}_\delta$ in $L^2(\Omega, \C^3)$
  along a subsequence. Once this is obtained, we have the desired
  contradiction: $\|\tilde E_\delta\|_{L^2(\Omega)} = 1$ is in
  conflict with the strong convergence
  $\tilde{E}_\delta \to \tilde{E} = 0$.
	
  \smallskip In order to show the strong convergence, we decompose
  $\tilde{E}_\delta$ according to the Helmholtz decomposition of Lemma
  \ref{lem:Helmholtz-dec}:
  $\tilde{E}_\delta = \tilde{E}_\delta^Y + \nabla \psi_\delta$ for
  $\tilde{E}_\delta^Y \in Y_\eps$ and $\psi_\delta \in H^1_0(\Omega)$.
  Boundedness of $\tilde{E}_\delta$ in
  $H_\Theta(\curl, \Omega, \Gamma)$ and boundedness of the projection
  implies that $\tilde{E}_\delta^Y$ is bounded in $Y_\eps$.  The
  compactness of $Y_\eps$ (shown in Lemma
  \ref{lem:MaxwellCompactness}) allows us to pass to a subsequence
  such that $\tilde{E}_\delta^Y$ converges strongly in
  $L^2(\Omega)$. We therefore have the desired result once we have the
  strong convergence of $\nabla\psi_\delta$ in $L^2(\Omega)$.
	
  \smallskip We use the test-function
  $\phi = \nabla\bar{\psi}_\delta$ in \eqref{eq:WeakForm:Maxwell-in-E-delta}.  Since the curl of a gradient
  vanishes and since we assumed that $f_e$ is orthogonal to gradients
  of $H^1_0(\Omega)$-function, we obtain
  \begin{align*}
    \int_\Omega (\eps + i\omega^{-2} \delta) \tilde{E}_\delta \cdot
    \nabla \bar{\psi}_\delta
    = 0\,.
  \end{align*}
  Inserting the decomposition
  $\tilde{E}_\delta = \tilde{E}_\delta^Y + \nabla \psi_\delta$ in the
  integral containing $\eps$, we find
  \begin{align}\label{eq:div-free:H-delta}
    i \int_\Omega \omega^{-2} \delta \tilde{E}_\delta \cdot
    \nabla \bar{\psi}_\delta
    + \int_\Omega \eps \nabla \psi_\delta\cdot\nabla \bar{\psi}_\delta
    = -\int_\Omega \eps \tilde{E}_\delta^Y \cdot
    \nabla \bar{\psi}_\delta
    = 0\,,
  \end{align}
  where we used the property $\tilde{E}_\delta^Y\in {Y_\eps}$
  in the last equality. The coercivity of $\eps$ allows us to deduce
  from \eqref{eq:div-free:H-delta}
  \begin{equation*}
    c_0 \|\nabla \psi_\delta\|_{L^2(\Omega)}^2
    \leq \int_\Omega \eps \nabla \psi_\delta \cdot\nabla \bar{\psi}_\delta
    = -i \int_\Omega \omega^{-2} \delta \tilde{E}_\delta \cdot
    \nabla \bar{\psi}_\delta\,.
  \end{equation*}
  With the Cauchy--Schwarz inequality and the normalization
  $\|\tilde{E}_\delta\|_{L^2(\Omega)} = 1$ we obtain
  $\|\nabla \psi_\delta\|_{L^2(\Omega)} \le C \delta \to 0$ as
  $\delta \to 0$.  This is the desired strong convergence of
  $\nabla \psi_\delta$.
	
  The strong convergence of $\tilde{E}_\delta^Y$ together with the
  strong convergence of $\nabla \psi_\delta$ implies the strong
  convergence of
  $\tilde{E}_\delta = \tilde{E}_\delta^Y + \nabla \psi_\delta$ in
  $L^2(\Omega)$. This provides the desired contradiction and concludes
  the proof.
\end{proof}

\bibliographystyle{plain}

\end{document}